\documentclass{amsart}

\usepackage{macros}
\standardsettings

\draftfalse

\begin{document}

\title[The real Hurwitz continued fraction expansion]{The Hurwitz continued fraction expansion as applied to real numbers}

\authordavid

\begin{Abstract}
Hurwitz (1887) defined a continued fraction algorithm for complex numbers which is better behaved in many respects than a more ``natural'' extension of the classical continued fraction algorithm to the complex plane would be. Although the Hurwitz complex continued fraction algorithm is not ``reducible'' to another complex continued fraction algorithm, over the reals the story is different. In this note we make clear the relation between the restriction of Hurwitz's algorithm to the real numbers and the classical continued fraction algorithm. As an application we reprove the main result of Choudhuri and Dani (2015).
\end{Abstract}

\maketitle

\section{Hurwitz's algorithm}
Let $x$ be a complex number such that $x\notin \Q(i)$. The \emph{(positive) Hurwitz continued fraction expansion} of $x$ (see \cite{Hensley_Hurwitz,Hensley_book,Hurwitz_CF}) is defined to be the expression
\begin{equation}
\label{posHurwitz}
a_0 + \cfrac{1}{a_1 + \cfrac{1}{a_2 + \ddots}},
\end{equation}
where the Gaussian integers $(a_n)_0^\infty$ (the \emph{partial quotients}) and the complex numbers $(x_n)_0^\infty$ are chosen recursively according to the \emph{Hurwitz algorithm}:
\begin{itemize}
\item $x_0 = x$.
\item If $x_n$ is defined, then $a_n$ is the Gaussian integer closest to $x_n$, which we denote by $[x_n]$.\Footnote{The tiebreaking mechanism is not relevant for the purposes of this paper, but for the sake of definiteness let us (agreeing with \cite{Hensley_Hurwitz}) set $[x] = [\Re x] + i [\Im x]$, where $[t]$ denotes the integer nearest to $t\in\R$, rounded down in the case of a tie.}
\item If $x_n$ and $a_n$ are both defined, then $x_{n + 1} = 1/(x_n - a_n)$.
\end{itemize}
It is not hard to see that the Hurwitz continued fraction expansion of $x$ always converges to $x$, and in fact the corresponding partial quotients are in some sense the ``best approximations possible'' \cite[Theorem 1]{Hensley_Hurwitz}.

Some authors \cite{ChoudhuriDani, KatokUgarcovici} also consider the \emph{negative Hurwitz continued fraction expansion} of a number $x$, which is the expression
\[
\w a_0 - \cfrac{1}{\w a_1 - \cfrac{1}{\w a_2 - \ddots}} = a_0 - \cfrac{1}{-a_1 - \cfrac{1}{a_2 - \ddots}},
\]
where $(a_n)_0^\infty$ are defined in the same way as in the positive Hurwitz continued fraction expansion, and $\w a_n = (-1)^n a_n$. Note that by the identity
\begin{equation}
\label{negative}
x + \frac{1}{y + z} = x - \frac{1}{-y - z},
\end{equation}
the convergents of the negative Hurwitz continued fraction expansion are the same as the convergents of the positive Hurwitz continued fraction expansion. Thus for many purposes, it is not necessary to distinguish between the positive and negative Hurwitz expansions.

In this note, we consider the restriction of the Hurwitz algorithm to the real line. In this case, it is clear that the numbers $(x_n)_0^\infty$ and $(a_n)_0^\infty$ will all be real. Moreover, unlike the case of the complex Hurwitz expansion, it is possible to say exactly when a sequence $(a_n)_0^\infty$ is the sequence of partial quotients of some real number:
\begin{proposition}\Footnote{This proposition is not original; the wording of \cite{ChoudhuriDani} seems to suggest that it was proven in the difficult-to-find \cite{KatokUgarcovici}.}
\label{propositioncharacterization}
For a sequence of integers $(a_n)_0^\infty$, the following are equivalent:
\begin{itemize}
\item[(A)] The expression \eqref{posHurwitz} is the Hurwitz continued fraction expansion of some (irrational) real number.
\item[(B)] For all $n\geq 1$, we have $|a_n| \geq 2$, with $a_n a_{n + 1} > 0$ if equality holds.
\item[(C)] For all $n\geq 1$, we have $|\w a_n| \geq 2$, with $\w a_n \w a_{n + 1} < 0$ if equality holds.
\end{itemize}
\end{proposition}
Obviously, (B) and (C) are reformulations of each other, so we prove (A) \iff (B):
\begin{proof}[Proof of \text{(A) \implies (B)}]
By definition, for all $n\geq 0$ we have $|x_n - a_n| \leq 1/2$ and thus $|x_{n + 1}| \geq 2$ and $|a_{n + 1}| \geq 2$. If equality holds, then $a_{n + 1}$ has the same sign as $x_{n + 1} - a_{n + 1}$, which in turn has the same sign as $a_{n + 2}$.
\end{proof}
\begin{proof}[Proof of \text{(B) \implies (A)}]
For each $n,N$ with $n\leq N$, let
\[
x_{n,N} = a_n + \cfrac{1}{\ddots + \cfrac{1}{a_N}}\,\cdot
\]
Reverse induction on $n$ shows that whenever $n\geq 1$, we have $|x_{n,N}| \geq 2$ and $a_n = [x_{n,N}]$. It follows that
\[
|x_{n,M} - x_{n,N}| \leq \frac{1}{4}|x_{n + 1,M} - x_{n + 1,N}| \;\;\text{ and }\;\; |x_{N,M} - x_{N,N}| \leq \frac{1}{2},
\]
which implies that $|x_{n,M} - x_{n,N}| \leq (1/4)^{\min(M,N) - n}$, and thus for each $n$ the limit
\[
x_n = \lim_{N\to\infty} x_{n,N} = a_n + \cfrac{1}{a_{n + 1} + \cfrac{1}{\ddots}}
\]
exists. We have $a_n = [x_n]$ and $x_{n + 1} = 1/(x_n - a_n)$, and thus \eqref{posHurwitz} is the Hurwitz continued fraction expansion of $x_0$.
\end{proof}

\section{Relation to the classical algorithm}
We now show that the restriction of the Hurwitz algorithm to the real line is in some sense ``equivalent'' to the classical continued fraction algorithm:

\begin{theorem}
\label{theorem1}
The sequence of convergents of the Hurwitz continued fraction expansion of a real number $x$ is a subsequence of the sequence of convergents of the classical continued fraction expansion of $x$. This sequence has the property that it omits no two consecutive convergents, and it also contains all rational approximants $p/q$ that satisfy the inequality $|x - p/q| \leq 1/(3q^2)$.
\end{theorem}
\begin{proof}
The key to the proof is the identity
\begin{equation}
\label{keyidentity}
\cfrac{1}{1 + \cfrac{1}{n + y}} = 1 - \frac{1}{n + 1 + y}\,,
\end{equation}
which is easily verified for all $n$ and $y$. Now let us denote the classical continued fraction expansion of a real number $x$ by
\begin{equation}
\label{classicalCF}
b_0 + \cfrac{1}{b_1 + \cfrac{1}{b_2 + \ddots}},
\end{equation}
so that $(b_n)_0^\infty$ is a sequence of integers and $b_n \geq 1$ for all $n \geq 1$. Let $S = \{n\geq 1 : b_n = 1\}$, and let $S'$ be the unique subset of $S$ with the following property:
\begin{itemize}
\item For all $n \in S$, we have either $n \in S'$ or $n - 1 \in S'$, but not both.
\end{itemize}
The set $S'$ can be constructed by taking each ``block'' of $S$ and selecting ``every other element'', starting from the first element of that block; for example, if $S = \{1,4,5,6,9,10\}$, then $S' = \{1,4,6,9\}$, since the ``blocks'' are $\{1\}$, $\{4,5,6\}$, and $\{9,10\}$.

For each $n\in S'$, in the expression \eqref{classicalCF} we replace
\[
\cfrac{1}{b_n + \cfrac{1}{b_{n + 1} + \ddots}} \;\;\;\; \text{ by } \;\;\;\; 1 -  \cfrac{1}{(b_{n + 1} + 1) + \ddots}
\]
according to \eqref{keyidentity}; this is possible since $b_n = 1$. This results in an expression of the form
\begin{equation}
\label{funnyHurwitz}
c_0 + \epsilon_0 + (-1)^{\epsilon_0} \cfrac{1}{c_1 + \epsilon_1 + (-1)^{\epsilon_1} \cfrac{1}{c_2 + \epsilon_2 + (-1)^{\epsilon_2} \cfrac{1}{\ddots}}},
\end{equation}
where $\epsilon_0,\epsilon_1,\cdots \in \{0,1\}$, and $c_n \geq 2$ for all $n\geq 1$. Here, we have used the facts that $\frac{1}{x} = 0 + (-1)^0 \frac{1}{x}$ and $1 - \frac{1}{x} = 1 + (-1)^1 \frac{1}{x}$ to represent the expressions $\frac{1}{x}$ and $1 - \frac{1}{x}$ in a uniform manner as $\epsilon + (-1)^\epsilon \frac{1}{x}$, where $\epsilon \in \{0,1\}$. Repeatedly applying the identity \eqref{negative} yields the Hurwitz expansion of $x$, so the convergents of \eqref{funnyHurwitz} are the same as the convergents of the Hurwitz expansion. But these are precisely those convergents $p_{n - 1}/q_{n - 1}$ of the classical expansion \eqref{classicalCF} such that $n\notin S'$. So the sequence of partial convergents of the Hurwitz expansion is a subsequence of the sequence of convergents of the classical expansion, which omits no two consecutive convergents (by the definition of $S'$). The omitted convergents are of the form $p_{n - 1}/q_{n - 1}$, where $n \in S'$, and these convergents satisfy
\[
\left|x - \frac{p_{n - 1}}{q_{n - 1}}\right| > \frac{1}{q_{n - 1} (q_n + q_{n - 1})} = \frac{1}{q_{n - 1} (b_n q_{n - 1} + q_{n - 2} + q_{n - 1})} > \frac{1}{(b_n + 2) q_{n - 1}^2} = \frac{1}{3 q_{n - 1}^2}
\]
(cf. \cite[Theorem 13]{Khinchin_book}). Here we have used the fact that $b_n = 1$ for all $n\in S'$. On the other hand, approximants that are not convergents of the classical expansion satisfy $|x - p/q| \geq 1/(2 q^2)$ \cite[Theorem 19]{Khinchin_book}. So all approximants that are not convergents of the Hurwitz expansion satisfy $|x - p/q| > 1/(3 q^2)$.
\end{proof}

Aside from the relation between the sequences of convergents described in Theorem \ref{theorem1}, the Hurwitz continued fraction expansion also shares the following formal similarity with the classical continued fraction expansion:
\begin{proposition}
\label{propositionrecursion}
Let $(a_n)_0^\infty$ be the sequence of Hurwitz (resp. classical) partial quotients of a real number $x$. If the sequences $(p_n)_{-2}^\infty$ and $(q_n)_{-2}^\infty$ are defined recursively via the formulas
\begin{align} \label{recursion1}
p_{-1} &= 1,\; p_{-2} = 0,& q_{-1} &= 0,\; q_{-2} = 1,\\ \label{recursion2}
p_n &= a_n p_{n - 1} + p_{n - 2},& q_n &= a_n q_{n - 1} + q_{n - 2},
\end{align}
then $(p_n/q_n)_0^\infty$ is precisely the sequence of Hurwitz (resp. classical) convergents of $x$.\Footnote{Note that in the Hurwitz case there is some ambiguity as to how to represent each convergent as a fraction ($p/q$ vs. $(-p)/(-q)$), and this proposition gives a way to resolve this ambiguity (namely to take the sequences $(p_n)_0^\infty$ and $(q_n)_0^\infty$ defined by the recursive relations). The ambiguity would be resolved in the same way if one took the expression defining the convergent and simplified it repeatedly according to the rules $(p/q)^{-1} = q/p$ and $n + p/q = (nq + p)/q$.}
\end{proposition}
\begin{proof}
The proof of \cite[Theorem 1]{Khinchin_book} is valid for both the classical and Hurwitz setups, since both use the same formal expressions for the convergents and partial quotients.
\end{proof}

However, there are differences from the classical algorithm as well. For example, while the error terms $p_n/q_n - x$ corresponding to the classical convergents always alternate in sign \cite[Theorem 4]{Khinchin_book}, the error terms corresponding to the Hurwitz convergents can be described as follows:

\begin{proposition}
If $p_n/q_n$ is the $n$th convergent of the Hurwitz algorithm, then the sign of the error term $p_n/q_n - x$ is the same as the sign of the $n$th partial quotient $\w a_{n + 1}$, i.e. $(-1)^{n + 1}$ times the sign of the $n$th partial quotient $a_{n + 1}$.
\end{proposition}
\begin{proof}
Since $x_{n + 1}$ and $a_{n + 1}$ share the same sign, comparing
\[
\frac{p_n}{q_n} = a_0 + \cfrac{1}{\ddots + \cfrac{1}{a_n + 0}} \;\;\;\;\text{ vs. }\;\;\;\; x = a_0 + \cfrac{1}{\ddots + \cfrac{1}{a_n + \cfrac{1}{x_{n + 1}}}}
\]
yields the desired conclusion.
\end{proof}

Another difference between the Hurwitz and classical expansions is that the Hurwitz expansion yields a faster rate of exponential growth for the denominators of the convergents. In the classical setup, the sequence $(q_n)_0^\infty$ always satisfies $\liminf_{n\to\infty} q_{n + 2}/q_n \geq 2 > 1$,\Footnote{In general, $q_{n + k}/q_n$ is always at least the $(k + 1)$st Fibonacci number. This is because if $F_k$ denotes the $k$th Fibonacci number, then an induction argument shows that $q_{n + k} = F_{k + 1} q_n + F_k q_{n - 1}$.} but it is possible that $\liminf_{n\to\infty} q_{n + 1}/q_n = 1$. By contrast:
\begin{proposition}
\label{propositionexpanding}
If $p_n/q_n$ denotes the $n$th convergent of the Hurwitz algorithm, then for all $n\geq 1$,
\begin{equation}
\label{qnanbounds}
(|a_n| - (2 - \phi))|q_{n - 1}| < |q_n| < (|a_n| + (\phi - 1))|q_{n - 1}|,
\end{equation}
where $\phi$ denotes the golden ratio. In particular,
\begin{equation}
\label{expanding}
|q_n| > \phi |q_{n - 1}|.
\end{equation}
\end{proposition}
\begin{proof}
For each $n\geq 1$ let $y_n = q_{n - 1}/q_n$. Then by \eqref{recursion1} and \eqref{recursion2}, we have $y_0 = 0$, and for all $n\geq 1$ we have
\[
y_n = \frac{1}{a_n + y_{n - 1}}\cdot
\]
By induction, $|y_{n - 1}| < 1$ for all $n\geq 1$, so $y_n$ shares the same sign as $a_n$. We will prove by induction that
\begin{equation}
\label{IH}
-(2 - \phi) < y_{n - 1} \sgn(a_n) < \phi - 1
\end{equation}
for all $n$. The base case $n = 1$ is trivial, so suppose that \eqref{IH} holds for some $n\geq 1$. Then
\begin{align*}
|y_n| &= \frac{1}{|a_n| + y_{n - 1}\sgn(a_n)} < \frac{1}{|a_n| - (2 - \phi)}
\leq \begin{cases}
\frac{1}{\phi} & |a_n| = 2\\
\frac{1}{1 + \phi} & |a_n| \geq 3
\end{cases}
=\begin{cases}
\phi - 1 & |a_n| = 2\\
2 - \phi & |a_n| \geq 3
\end{cases}\cdot
\end{align*}
To complete the inductive step, we need to show that if $|a_n| = 2$, then $y_n \sgn(a_{n + 1}) > 0$. But this follows from Proposition \ref{propositioncharacterization}, since $\sgn(y_n) = \sgn(a_n)$.

Combining \eqref{IH} with the formula
\[
|q_n| = (|a_n| + y_{n - 1}\sgn(a_n)) |q_{n - 1}|
\]
demonstrates \eqref{qnanbounds}. Finally, the inequality $|a_n| \geq 2$ gives \eqref{expanding}.
\end{proof}

\section{Relation with Diophantine approximation}
Although the connection between the classical continued fraction expansion of a real number $x$ and the Diophantine properties of $x$ has been dealt with extensively in a number of places, the connection with the Hurwitz algorithm has not been stated precisely before. Many results can be proven simply from the identification of the Hurwitz convergent sequence with a subsequence of the classical convergent sequence, i.e. Theorem \ref{theorem1}. For brevity we do not list these here. One place where a difference does appear is in the basic estimates for the accuracy of the approximation of a convergent. In the classical setting, we have
\[
\frac{1}{(b_n + 2) q_{n - 1}^2} < \left|x - \frac{p_{n - 1}}{q_{n - 1}}\right| < \frac{1}{b_n q_{n - 1}^2}
\]
(e.g. this follows from \cite[Theorems 9 and 13]{Khinchin_book}). By contrast, in the Hurwitz setup we have:
\begin{proposition}
\label{propositionlagrangebound}
If $p_n/q_n$ denotes the $n$th convergent of the Hurwitz expansion of $x$, then
\[
\frac{1}{(|a_n| + (\phi - 0.5)) q_{n - 1}^2} < \left|x - \frac{p_{n - 1}}{q_{n - 1}}\right| < \frac{1}{(|a_n| - (2.5 - \phi)) q_{n - 1}^2}\cdot
\]
\end{proposition}
\begin{proof}
By \cite[Theorem 5]{Khinchin_book}, we have
\[
x = \frac{x_n p_{n - 1} + p_{n - 2}}{x_n q_{n - 1} + q_{n - 2}},
\]
where $x_n$ is as in the definition of the Hurwitz algorithm, i.e.
\[
x_n = a_n + \cfrac{1}{a_{n + 1} + \ddots}\,\cdot
\]
Thus
\begin{align*}
q_{n - 1}^2 \left|x - \frac{p_{n - 1}}{q_{n - 1}}\right|
&= q_{n - 1} \left|\frac{(q_{n - 1} x_n p_{n - 1} + q_{n - 1} p_{n - 2}) - (p_{n - 1} x_n q_{n - 1} + p_{n - 1} q_{n - 2})}{x_n q_{n - 1} + q_{n - 2}}\right|\\
&= q_{n - 1} \frac{1}{|x_n q_{n - 1} + q_{n - 2}|} = \frac{1}{|x_n + q_{n - 2}/q_{n - 1}|} = \frac{1}{|x_n - a_n + q_n/q_{n - 1}|}\cdot
\end{align*}
Since $|x_n - a_n| \leq 1/2$, combining with \eqref{qnanbounds} completes the proof.
\end{proof}

\section{Comparison with Choudhuri and Dani (2015)}
In this section we show that by combining the results of previous sections in an appropriate way, we can strengthen a result of Choudhuri and Dani \cite{ChoudhuriDani}. We state and prove our theorem below and then show that it implies the main result of \cite{ChoudhuriDani}.
% In this section we prove a theorem that captures the spirit of the main theorem of \cite{ChoudhuriDani}, i.e. \cite[Theorem 1.1]{ChoudhuriDani}. The reader should be aware that the statement of \cite[Theorem 1.1]{ChoudhuriDani} contains a few inaccuracies. For details we refer to the erratum of \cite{ChoudhuriDani} (currently unpublished, but available from the authors upon request).

%\begin{itemize}
%\item An extraneous factor of 2 appears in the definitions of $e(\delta)$ and $f(\delta)$; this factor does not appear in the proof, leading to the conclusion that it should be considered a typo.
%\item Part (ii) is phrased incorrectly; the authors of \cite{ChoudhuriDani} have suggested replacing it with ``Given $m>f(\delta)$ and $0< M \leq \max \{\frac {1}{4}\log \frac {9}{5}, \frac {1}{8}\alpha^-\}$ there exists ...'' (private communication)
%\end{itemize}
%Therefore, determining whether the result below is really more powerful or not may be difficult and/or tedious.

\begin{theorem}
\label{theorem2}
Let $(a_n)_0^\infty$ be the Hurwitz partial quotient sequence of a real number $x$, and fix $0 < \delta \leq 1/3$. For each $\rho > 0$, let
\begin{equation}
\label{Xrhodef}
X_\rho = \frac{\#\{(p,q)\in\mathbb Z^2 \text{ primitive} : 0 < q \leq \rho, \; |q(qx - p)| \leq \delta\}}{\log(\rho)}\cdot
\end{equation}
Then
\begin{align} \label{21}
\liminf_{\rho\to\infty} X_\rho &\geq \liminf_{n\to\infty} \frac{\#\big\{j = 1,\ldots,n : |a_{j + 1}| \geq \delta^{-1} + (2.5 - \phi)\big\}}{\sum_{j = 1}^n \log(|a_j| + (\phi - 1))}\\ \label{22}
\limsup_{\rho\to\infty} X_\rho &\leq \limsup_{n\to\infty} \frac{\#\big\{j = 1,\ldots,n : |a_{j + 1}| \geq \delta^{-1} - (\phi - 0.5)\big\}}{\sum_{j = 1}^n \log(|a_j| - (2 - \phi))}
\end{align}
\end{theorem}
\begin{proof}
By Theorem \ref{theorem1}, the condition $\delta \leq 1/3$ implies that the set appearing in \eqref{Xrhodef} contains only pairs $(p,q)$ such that $p/q$ is a convergent of the Hurwitz expansion of $x$. Thus the numerator of \eqref{Xrhodef} is constant with respect to $\rho$ along intervals of the form $(|q_{n - 1}|,|q_n|)$, and increases by at most 1 from $|q_n| - o(1)$ to $|q_n| + o(1)$. It follows that $\liminf_{\rho\to\infty} X_\rho = \liminf_{n\to\infty} X_{|q_n|}$, and similarly for the limsup. Now, applying Theorem \ref{theorem1} again, we have
\[
X_{|q_n|} = \frac{\#\{j = 1,\ldots,n : |q_j (q_j x - p_j)| \leq \delta\}}{\log|q_n|}\cdot
\]
To finish the proof, we have to bound this expression between the corresponding expressions in the right hand sides of \eqref{21} and \eqref{22}. And indeed, by Proposition \ref{propositionlagrangebound} we have
\[
|a_{j + 1}| \geq \delta^{-1} + (2.5 - \phi) \;\;\Rightarrow\;\; |q_j (q_j x - p_j)| \leq \delta \;\;\Rightarrow\;\; |a_{j + 1}| \geq \delta^{-1} - (\phi - 0.5)
\]
and thus
\begin{align*}
\#\{j = 1,\ldots,n : |a_{j + 1}| \geq \delta^{-1} + (2.5 - \phi)\}
&\leq \#\{j = 1,\ldots,n : |q_j (q_j x - p_j)| \leq \delta\}\\
&\leq \#\{j = 1,\ldots,n : |a_{j + 1}| \geq \delta^{-1} - (\phi - 0.5)\}.
\end{align*}
On the other hand, iterating \eqref{qnanbounds} and taking logarithms gives
\[
\sum_{j = 1}^n \log(|a_j| - (2 - \phi)) \leq \log|q_n| \leq \sum_{j = 1}^n \log(|a_j| + (\phi - 1))
\]
and dividing these two pairs of inequalities completes the proof.
\end{proof}

%\begin{remark*}
%Among the many differences between Theorem \ref{theorem2} and \cite[Theorem 1.1]{ChoudhuriDani}, one that should not cause too much confusion is the the fact that in \cite{ChoudhuriDani}, the expression $q(qx - p)$ is replaced by a general indefinite quadratic form $Q(p,q) = (aq + bp)(cq + dp)$ such that $ad - bc = 1$. This does not increase the level of generality at all, since any quadratic form $Q$ of this form can be written as $Q(p,q) = (qx - p)(q + y(qx - p))$ for some real numbers $x$ and $y$, and the contribution from the term $y(qx - p)^2$ becomes negligible as $\rho\to\infty$, at least for those points $(p,q)$ such that $|qx - p| \leq |q + y(qx - p)|$. The set of points $(p,q)$ for which the opposite inequality occurs can be dealt with similarly (using different values of $x$ and $y$).
%\end{remark*}

We now show that Theorem \ref{theorem2} implies the main result of \cite{ChoudhuriDani}. Since the statement of the main theorem of that paper contains a few inaccuracies, we state a corrected version here, which is equivalent to the version that appears in the authors' erratum (currently unpublished, but available from the authors upon request).

\begin{theorem}[Corrected version of {\cite[Theorem 1.1]{ChoudhuriDani}}]
Let $Q(p,q) = (aq + bp)(cq + dp)$ be a quadratic form, where $a,b,c,d\in \R$, $ad -bc = 1$, $b\neq 0$, and $\frac ab$ is irrational. Let $(a_n)_0^\infty$ be the Hurwitz partial quotient sequence of $\frac ab$. Let
\begin{align*}
\alpha^- &= \liminf_{n\to\infty} \frac{1}{n} \sum_{j = 1}^n \log|a_j|,&
\alpha^+ &= \limsup_{n\to\infty} \sum_{j = 1}^n \log|a_j|.
\end{align*}
For each $A > 0$ let
\begin{align*}
D^-(A) &= \liminf_{n\to\infty} \frac{1}{n} \#\{j = 1,\ldots,n : |a_{j + 1}| \geq A\}\\
 D^+(A) &= \limsup_{n\to\infty} \frac{1}{n} \#\{j = 1,\ldots,n : |a_{j + 1}| \geq A\}.
\end{align*}
Fix $0 < \delta < \frac 1\pi$, and let $e(\delta) = D^-(\delta^{-1} + 1)$ and $f(\delta) = D^+(\delta^{-1} - \frac 32)$. Let $\kappa > 0$ be fixed and for each $\rho > 0$ let
\[
G(\rho) = \{(p,q)\in \Z^2 \text{ primitive} : 0 < |Q(p,q)| < \delta, \;\; cq + dp > \kappa, \;\; \|(p,q)\| \leq \rho\}.
\]
Then we have the following:
\begin{itemize}
\item[(i)] if $\alpha^+ < \infty$ then there exists $\rho_0$ such that for all $\rho\geq \rho_0$ we have
\[
\#G(\rho) \geq \frac{e(\delta)}{\alpha^+ + 3} \log(\rho);
\]
\item[(ii)] Let $M = \max(\frac 14\log(\frac 95),\frac 18\alpha^-)$ if $\alpha^- < \infty$, and let $M < \infty$ be arbitrary if $\alpha^- = \infty$. Then for any $m > f(\delta)$, there exists $\rho_0$ such that for all $\rho \geq \rho_0$ we have
\[
\#G(\rho) \leq \frac{m}{M}\log(\rho).
\]
\end{itemize}
\end{theorem}
\begin{proof}[Proof using Theorem \ref{theorem2}]
Let $x = -\frac ab$ and $y = bd$. Since $ad - bc = 1$, we have
\[
Q(p,q) = (qx - p)(q + y(qx - p))
\]
and thus
\begin{equation}
\label{Qpqasymp}
\lim_{\substack{(p,q) \in G(\infty) \\ \|(p,q)\| \to \infty}} \frac{|Q(p,q)|}{|q(qx - p)|} =
\lim_{\substack{(p,q)\in \Z^2 \\ |q(qx - p)| \leq 1 \\ q\to\infty}} \frac{|Q(p,q)|}{|q(qx - p)|} =
1.
\end{equation}
Moreover, the Hurwitz partial quotient sequence of $x$ is the same as the Hurwitz partial quotient sequence of $\frac ab$ except for minus signs.

Fix $0 < \delta < \frac 1\pi$. We prove (i) and (ii):
\begin{itemize}
\item[(i)] Since $1 > 2.5 - \phi$, there exists $0 < \w\delta < \delta < 1/\pi < 1/3$ such that $\delta^{-1} + 1 \geq \w\delta^{-1} + (2.5 - \phi)$. It follows that
\[
e(\delta) = D^-(\delta^{-1} + 1) \leq D^-(\w\delta^{-1} + (2.5 - \phi)).
\]
Now by \eqref{Qpqasymp}, we have
\[
X_\rho(\w\delta) \leq \#G(\rho) + C
\]
for some constant $C$ depending on $\delta$ and $\w\delta$. Thus if $\alpha^+ < \infty$, then
\begin{align*}
\liminf_{\rho\to\infty} \frac{\#G(\rho)}{\log(\rho)}
&\geq \liminf_{\rho\to\infty} X_\rho(\w\delta)
\underset{\eqref{21}}{\geq} \frac{D^-(\w\delta^{-1} + (2.5 - \phi))}{\alpha^+ + (\phi - 1)}
> \frac{e(\delta)}{\alpha^+ + 3},
\end{align*}
which implies (i).
\item[(ii)] Since $\frac{3}{2} > \phi - 0.5$, there exists $0 < \delta < \w\delta < 1/\pi < 1/3$ such that $\delta^{-1} - \frac{3}{2} \leq \w\delta^{-1} - (\phi - 0.5)$. It follows that
\[
f(\delta) = D^+(\delta^{-1} - \tfrac{3}{2}) \geq D^+(\w\delta^{-1} - (\phi - 0.5)).
\]
Now by \eqref{Qpqasymp}, we have
\[
X_\rho(\w\delta) \geq \#G(\rho) - C
\]
for some constant $C$ depending on $\delta$ and $\w\delta$. Thus
\begin{align*}
\liminf_{\rho\to\infty} \frac{\#G(\rho)}{\log(\rho)}
&\leq \liminf_{\rho\to\infty} X_\rho(\w\delta)
\underset{\eqref{22}}{\leq} \frac{D^+(\w\delta^{-1} - (\phi - 0.5))}{\alpha^- - (2 - \phi)}
\leq \frac{f(\delta)}{\max(\frac 14 \log(\frac 95),\frac 18 \alpha^-)},
\end{align*}
which implies (ii). In the last inequality, we have used the bound
\[
\alpha^- - (2 - \phi) \geq \max\left(\tfrac 14 \log(\tfrac 95),\tfrac 18 \alpha^-\right),
\]
which follows from the fact that $\alpha^- \geq \log(2)$ (cf. Proposition \ref{propositioncharacterization}) together with the numerical bound
\[
\log(2) - (2 - \phi) > \max\left(\tfrac 14 \log(\tfrac 95),\tfrac 18 \log(2)\right).
\qedhere\]
\end{itemize}
\end{proof}

{\bf Acknowledgements.} The author was supported by the EPSRC Programme Grant EP/J018260/1. The author thanks the anonymous referee for helpful comments.

\bibliographystyle{amsplain}

\bibliography{bibliography}

\end{document}